\documentclass[12pt]{extarticle}
\usepackage{amsmath, amsthm, amssymb, color}
\usepackage[colorlinks=true,linkcolor=blue,urlcolor=blue,citecolor=olive]{hyperref}
\usepackage{graphicx}
\usepackage{caption}
\usepackage{mathtools}
\usepackage{mathdots}
\usepackage{enumerate}
\usepackage{verbatim}
\usepackage{tikz,tikz-cd,tikz-3dplot}
\usepackage{amssymb}
\usetikzlibrary{matrix}
\usetikzlibrary{arrows}
\usepackage{algorithm}
\usepackage{caption}
\usepackage{subcaption}
\oddsidemargin \evensidemargin
\usepackage{cleveref}

\newtheorem{theorem}{Theorem}
\newtheorem{proposition}[theorem]{Proposition}
\newtheorem{lemma}[theorem]{Lemma}
\newtheorem{corollary}[theorem]{Corollary}

\theoremstyle{definition}
\newtheorem{definition}[theorem]{Definition}

\newtheorem{remark}[theorem]{Remark}
\newtheorem{conjecture}[theorem]{Conjecture}

\newtheorem{example}[theorem]{Example}

\numberwithin{theorem}{section}

\newcommand{\PP}{\mathbb{P}}

\newcommand{\QQ}{\mathbb{Q}}
\newcommand{\CC}{\mathbb{C} }

\newcommand{\NN}{\mathbb{N}}

\newcommand\exend{\hfill$\blacksquare$}

\title{\bf Taylor Polynomials of Rational Functions}

\author{Aldo Conca, Simone Naldi, Giorgio Ottaviani and Bernd Sturmfels}

\date{}
\begin{document}

\maketitle
 \begin{abstract}
 \noindent
A Taylor variety consists of all fixed order Taylor polynomials 
of rational functions,
where the number of variables and degrees
of numerators and denominators are~fixed.
In one variable, Taylor varieties are given by
rank constraints on Hankel matrices.
Inversion of the natural parametrization is known as Pad\'e approximation.
We study the dimension and defining ideals of Taylor~varieties.
Taylor hypersurfaces are interesting for projective geometry, since
their Hessians tend to vanish.
In three and more variables, there exist defective Taylor varieties
whose dimension is smaller than the number of parameters.
We explain this with Fr\"oberg's Conjecture in
commutative algebra. \end{abstract}

\section{Introduction}

Given two polynomials $P$ and $Q$ whose constant term is $1$, the
rational function $P/Q$ has a Taylor series expansion with constant term $1$.
Truncating that series at terms of degree $m$, we obtain the $m$th  Taylor polynomial of $P/Q$.
Its coefficients are polynomials in the coefficients of $P$ and $Q$. For example, consider the
fifth Taylor polynomial of two univariate~quadrics:
\begin{equation} \label{eq:intro1} \frac{1 + p_1 x + p_2 x^2}{1 + q_1 x + q_2 x^2}\,\,\, = \,\, \begin{small}
\begin{matrix} 1+
(p_1{-}q_1)x
- (p_1q_1{-}q_1^2{-}p_2{+}q_2) x^2
+ (p_1 q_1^2{-}q_1^3{-}p_1 q_2{-}p_2 q_1{+}2 q_1 q_2)x^3 \\
-\, (p_1q_1^3-q_1^4-2p_1q_1q_2- p_2q_1^2+3q_1^2 q_2+p_2 q_2-q_2^2)x^4  
+(p_1 q_1^4  -q_1^5 \\ -3p_1q_1^2q_2-p_2q_1^3+4q_1^3q_2+p_1q_2^2+2p_2q_1q_2-3q_1q_2^2)x^5
\,+ \, \ldots
\end{matrix} \end{small}
\end{equation}
This rational function has four parameters $p_1,p_2,q_1,q_2$, but
 the quintic has five coefficients:
\begin{equation} \label{eq:intro2}
  1 \,+\, c_1 x \,+\, c_2 x^2 \,+ \,c_3 x^3 \,+\, c_4 x^4 \,+\, c_5 x^5 \,+\, \cdots.
  \end{equation}
  Therefore, if (\ref{eq:intro2}) is a Taylor polynomial
(\ref{eq:intro1}), then its coefficients $c_i$ must satisfy a constraint.

To find this constraint, we equate (\ref{eq:intro2}) with the right hand side of (\ref{eq:intro1}), and we
extract the coefficients of all $x$-monomials. This yields a system of
five polynomial equations. From that system we eliminate the four unknowns $p_1,p_2,q_1,q_2$.
This process leads to the cubic~equation
\begin{equation} \label{eq:intro3}
{\rm det} \begin{bmatrix}
c_5 & c_4 & c_3 \\
c_4 & c_3 & c_2 \\
c_3 & c_2 & c_1  \end{bmatrix} 
\,\, = \,\,
c_1 c_3 c_5-c_1c_4^2-c_2^2 c_5+2c_2c_3 c_4-c_3^3 \quad = \quad 0.
\end{equation}
This is the condition for a quintic to be a Taylor polynomial for the ratio of two quadrics.

For a geometric view, note that (\ref{eq:intro1})  defines a polynomial map from the
$4$-space with coordinates $(p_1,p_2,q_1,q_2)$ into the $5$-space
with coordinates $(c_1,c_2,c_3,c_4,c_5)$. The closure of the image of this map is called
Taylor variety (see \Cref{taylorvariety}). For the example in (\ref{eq:intro1}), the Taylor variety 
is the hypersurface in projective $4$-space $\PP^4$ defined by the equation (\ref{eq:intro3}).
In this article we study such algebraic varieties of truncated rational series  in $n \geq 1$ variables.

The general case concerns a rational function $P/Q$
where $P$ and $Q$ are polynomials in  $x=(x_1, \ldots, x_n)$
that satisfy $P(0) = Q(0) = 1$. We assume that $P$ has degree $\leq d$
and $Q$ has degree $\leq e$.
We consider the Taylor series of this rational function, expanded up to order $m$:
\begin{equation}
\label{eq:Pade}
\qquad \frac{P(x)}{Q(x)}
\quad =
\quad
\sum_{|\gamma| \leq m}  c_{\gamma} \, x^{\gamma} 
\qquad + \quad \hbox{terms of order $\,\geq m+1$}.
\end{equation}
Here $x^\gamma$ denotes the monomial
$x_1^{\gamma_1} \cdots x_n^{\gamma_n}$ and $|\gamma| = \gamma_1+\cdots+\gamma_n$ is the total
degree of $x^\gamma$.

\smallskip

The main point of this paper is the characterization of all  polynomials
$T=\sum c_\gamma x^\gamma$
that admit such an approximation.
The numerator  $P$ and the denominator $Q$  in  \eqref{eq:Pade}
have degrees $d$ and $e$ respectively.
We call (\ref{eq:Pade}) a {\it Padé approximation} of type $(d,e)$.
Such approximations of analytic functions
originated in work of Hermite \cite{H} and Pad\'e \cite{Pade}.
This topic belongs to numerical analysis, where it is studied mostly in the univariate case \cite{B}.
The computation of a single Pad\'e approximant 
is an instance of a block Toeplitz linear algebra problem \cite{brent}.
Pad\'e approximation is also a classical question of computer algebra; see {\it e.g.}~\cite{BL}.
The Pad\'e approximation problem for $n \geq 2$ can be interpreted
as a problem of computing syzygies~\cite{NN1}.

If the number of parameters ({\it i.e.}~the coefficients of $P$ and $Q$) is small
compared to the approximation order $m$, and if $n=1$, then the rational function $P/Q$ is uniquely determined
by its Taylor polynomial $T$. This is noted {\it e.g.} in \cite[Section~1.1]{B}.
For us, this means that the parametrization $(P,Q) \mapsto T$ is an injective map.
This holds for $n=1$ but fails
for $n \geq 3$. We shall see that the fibers of the map $(P,Q) \mapsto T$ can 
be positive-dimensional, {\it i.e.}~the
uniqueness of the
Pad\'e approximation breaks down.
The geometric study in this article thus represents
a foundational contribution to the theory of
 multivariate Pad\'e approximations.

All students of calculus are familiar with Taylor series expansions.
We therefore chose the name Taylor instead of the name Pad\'e
for the geometric object we shall investigate.
\begin{definition}
  \label{taylorvariety}
  The {\em Taylor variety} $\mathcal{T}^n_{d,e,m}$ is defined as the closure in
  $\PP^{{\small \binom{n+m}{n}-1}}$
  of the set of Taylor polynomials of degree $\leq m$ of rational functions \eqref{eq:Pade} of degree
  $(d,e)$ in $n$ variables.
\end{definition}

The projective space $\PP^{{\small \binom{n+m}{n}-1}}$ 
comprises polynomials in $n$ variables of degree $\leq m$, up to scaling.
Polynomials with $c_0 = 1$ form an affine open chart 
$\CC^{\binom{n+m}{n}-1}$.
We here work over~$\CC$, but our theory
extends to all fields. Our computations are done over the rational numbers~$\QQ$.

The Taylor variety $\,\mathcal{T}^n_{d,e,m}$ is irreducible since it arises
from the image of a polynomial map.
A natural first question is: what is its dimension?
Let's start by counting parameters.
The polynomial $P$ has $\binom{n+d}{n}-1$ free coefficients,
and the polynomial $Q$ has $\binom{n+e}{n}-1$ free coefficients.
Since the dimension can never increase under a polynomial map, we conclude
\begin{equation}
\label{eq:expecteddimension}
\begin{matrix}
{\rm dim}\bigl(\mathcal{T}^n_{d,e,m} \bigr) \,\, \leq \,\,
  {\rm min}
\bigl\{\, \binom{d+n}{n}+\binom{e+n}{n}-2\,,\,
\binom{m+n}{n}-1 \, \bigr\} .  \end{matrix} 
\end{equation}
The quantity on the right hand side is the {\em expected dimension} of the Taylor variety 
$\mathcal{T}^n_{d,e,m}$.
In our example (\ref{eq:intro1}), which is the case $n=1,d=e=2,m=5$,
 the expected dimension is $4$. And, indeed,
$\mathcal{T}^1_{2,2,5}$ is the cubic hypersurface in $\PP^5$ defined by
 (\ref{eq:intro3}), so its dimension equals $4$.
 
 \smallskip
 
 We now discuss the organization of this paper and we summarize our results.
 Section~\ref{sec2} resolves the univariate case $(n=1)$.
Here the dimension equals the expected dimension.
 The variety  $\mathcal{T}^1_{d,e,m}$ lives in $\PP^m$ and it has dimension
 $d+e$, provided $d+e < m$. Its prime ideal is generated by the maximal minors of a
 Hankel matrix with $m-d$ rows and $e+1$ columns (Theorem \ref{thm:ideal:minors}).
Computing the kernel of this matrix is a key step for Pad\'e approximation.
 
In Section \ref{sec3} we turn to $n \geq 2$. We
introduce the Pad\'e matrix, which has a block Hankel structure.
Its ideal of maximal minors is generally not radical and can have multiple irreducible
components. Among them is the Taylor variety $\mathcal{T}^n_{d,e,m}$.
In Theorem  \ref{IdealPrime} we identify its prime ideal $\mathcal{I}^n_{d,e,m}$.
Computationally-minded readers can jump to Example  \ref{ex:2124} right now.
 
Our key finding is that Taylor varieties can be {\em defective}, {\em i.e.}~the
 inequality in (\ref{eq:expecteddimension}) is strict.
The smallest instance ($n=3,d=e=2,m=3$)
is worked out in detail in Proposition~\ref{prop:degenhyper}.

In Section \ref{sec4} we focus on Taylor varieties that are defined by a single polynomial,
so they have codimension one in their ambient space.
Some of these {\em Taylor hypersurfaces} exhibit a property that is of interest 
in projective geometry, namely their Hessian vanishes identically.

In Section \ref{sec5}  we derive a general formula for the dimension of
the Taylor variety $\mathcal{T}^n_{d,e,m}$.
This enables the computations, reported in
Table \ref{tab:defective}, which culminate in Conjecture \ref{conj:finite}.

In Section~\ref{sec6} we recast our dimension formula in terms of
Hilbert functions of ideals of generic forms.
This yields a link to
Fr\"oberg's Conjecture, which is a longstanding open problem
in commutative algebra. In  Theorem \ref{thm:FRob} and in Corollary \ref{cor(1)(2)},  this connection is used to prove Conjecture \ref{conj:finite} in some special cases. 
 We are grateful to
 Christian Krattenthaler for  suggesting the  proof of  Theorem \ref{thm:FRob}  and for allowing us to include it  in our paper.

\section{One Variable}
\label{sec2}

We consider polynomials in one variable $x$ of the form
$T = 1 + c_1 x + c_2 x^2 + \cdots + c_m x^m$.
The set of these polynomials is identified with the
vector space $\CC^m$ with coordinates $(c_1,c_2,\ldots,c_m)$.
Similarly, we identify $\CC^d$ and $\CC^e$ respectively with  the spaces of
 polynomials $P$ of degree $\leq d$ and 
 $Q$ of degree $\leq e$ such that $P(0) = Q(0) = 1$.
 We are interested in the polynomial map
 \begin{equation}
 \label{eq:mappsi}
\psi \,: \, \CC^d \times \CC^e\, \rightarrow \,\CC^m  \,,\,\, (P,Q) \,\mapsto \,\,
\hbox{the order $m$ Taylor polynomial of $\,P/Q $}. 
\end{equation}

We fix the projective space $\PP^m$ with coordinates $(c_0:c_1:\cdots:c_m)$.
The Taylor variety $\mathcal{T}^1_{d,e,m}$ is the closure in $\PP^m$ of the image of $\psi$.
In words, $\mathcal{T}^1_{d,e,m}$ 
is the smallest projective variety containing all polynomials $T$ of degree~$m$
whose Pad\'e approximation of type~$(d,e)$ is exact. We first exclude the trivial case when
the Taylor variety fills its ambient projective space.

\begin{lemma}
  \label{lem:dense}
Assume $d+e \geq m$. Then $\mathcal{T}^1_{d,e,m}=\PP^m$.
\end{lemma}
\begin{proof}
  We assume $d+e = m$ and claim that the map $\psi$ in \eqref{eq:mappsi} is dominant.
  This covers the case $d+e \geq m$ since $d$ and $e$ are upper bounds on the degrees.
  We note that the product $QT=(1+\sum_{i=1}^e q_i x^i)(\sum_{i=0}^m c_i x^i)$
  is a polynomial of degree    $\leq d$ modulo $x^{m+1}$ if and only if
  \begin{equation}\label{eq_system_lemma3}
  \begin{bmatrix}
  c_{m-1} & c_{m-2} & \cdots & c_d \\
  c_{m-2} & c_{m-3} & \cdots.& c_{d-1} \\
   \vdots & \vdots & \ddots & \vdots \\
   c_d & c_{d-1} & \cdots & c_{d-e+1} 
     \end{bmatrix} \cdot
  \begin{bmatrix}     q_1 \\  q_2 \\   \vdots \\     q_e   \end{bmatrix}
\,  \,\,=\,\,\,
  -   \begin{bmatrix}     c_m \\  c_{m-1} \\   \vdots \\     c_{d+1}   \end{bmatrix}  ,
  \end{equation}
  where $c_i=0$ if $i<0$. Let $\mathcal{U} \subset \PP^m$ be the
  non-empty Zariski open set of all polynomials $T$ such that the matrix on the left side of \eqref{eq_system_lemma3}
  is non-singular. Every $T \in \mathcal{U}$ has a unique 
  exact Pad\'e approximation $P/Q$ of type
  $(d,e)$. This shows that $\mathcal{U} \subseteq\psi(\CC^d \times \CC^e)$, as claimed.
\end{proof}

From now on we assume $d+e < m$. 
From the proof of \Cref{lem:dense}, one deduces that the Pad\'e approximation is unique when
$T \in \mathcal{T}^1_{d,e,m}$ is generic.
The map $\psi$ in \eqref{eq:mappsi} is birational onto its image. In  particular, the Taylor variety    $\mathcal{T}^1_{d,e,m}$
has the expected dimension $d+e$.

Fix the set of monomials  $M_{d+1,m} = \{x^{d+1},\ldots,x^{m}\}$.
Multiplication by a polynomial $T = 1 + c_1 x + \cdots + c_m x^m$
   defines a linear map
  from $\CC[x]_{\leq e}$ to $\CC[x]_{\leq e+m}$. Composing this map with the
  projection onto the linear span of $M_{d+1,m}$ in the polynomial ring $\CC[x]$ yields 
  \begin{equation}
  \label{eq:varphi1}
  \begin{array}{cccccl}
    \varphi_T : & \CC[x]_{\leq e} & \to     & \CC[x]_{\leq e+m} & \to     & \CC \{M_{d+1,m}\}, \\
    &               Q & \mapsto & QT                & \mapsto & QT \,\,\, \hbox{restricted to $M_{d+1,m}$.}
  \end{array}
\end{equation}
We fix the monomial bases for both the domain and the image.
For the domain we order that monomial basis by increasing degree, and for
the image we order it by decreasing degree.
With this convention,
    the  $(m-d) \times (e+1)$ matrix that
   represents the linear map $\varphi_T$ equals
   \begin{equation}
   \label{eq:convention}
  P_T \,\, = \,\,
\begin{bmatrix}
c_m & c_{m-1} & \cdots & c_{m-e} \\
c_{m-1} & c_{m-2} & \cdots & c_{m-e-1} \\
 \vdots & \vdots & \ddots & \vdots \\
c_{d+2} & c_{d+1} & \cdots & c_{d-e+2} \\
 c_{d+1} & c_{d} & \cdots & c_{d-e+1}
\end{bmatrix}.
\end{equation}
Here $d-e+1$ is allowed to be negative,
and we set $c_i = 0$ whenever $i < 0$.
By our assumption,
the number $m-d$ of rows of $P_T$ is greater or equal to the number
$e+1$ of columns of $P_T$.

\begin{example}[$d=0, e=m-1$]
This extreme case is the variety of
 reciprocals of polynomials, expanded up to
one order beyond their degree. Here the matrix $P_T$ is
square, and the last row and the column have only two
non-zero entries. For instance, for $m=4$ we have
$$ P_T \,\,= \,\, \begin{bmatrix} 
 c_4 & c_3 & c_2 & c_1 \\
  c_3 & c_2 & c_1 & c_0 \\
 c_2 & c_1 & c_0 & 0 \\
 c_1 & c_0 &  0  & 0 \\
\end{bmatrix} .
$$
The determinant of this $m \times m$ matrix is irreducible.
Our next result generalizes~this.  \exend 
\end{example}

We now turn to some commutative algebra
 in the polynomial ring $\CC[c] = \CC[c_0,c_1,\ldots,c_m]$. For a polynomial matrix $X$  and $s\in \NN$  we will denote by $I_s(X)$ the  ideal generated by all 
  $s \times s$ minors of $X$, and we set  $I_{\rm max}(X)=I_r(X)$
 when $r$ denotes the rank of $X$.  

\begin{theorem}
  \label{thm:ideal:minors}
  The ideal $I_{e+1}(P_T)$ generated by the $\,(e+1) \times (e+1)$ minors of the matrix 
 in (\ref{eq:convention})  is prime in $\CC[c]$.
  It defines an irreducible projective  variety of dimension $d+e$ in~$\PP^m$.
\end{theorem}

\begin{proof}
 A   generic Hankel matrix is $1$-generic in the sense of  Eisenbud \cite{E1, E2}. 
 Indeed, in  \cite[Theorem 1]{E1},
  Eisenbud proves a beautiful result which states that  the ideal of maximal minors of a $1$-generic  matrix of any format $a\times b$ (with $a\geq b$) is prime of the expected codimension $a-b+1$,
 and the same statement is true for any linear section of codimension $\leq b-2$.

 If $e \leq d+1$ then $P_T$ is a generic Hankel matrix and our conclusion follows immediately.
 Suppose now that $e >d+1$. The $(m-d) \times (e+1)$ matrix $P_T$ is obtained from  the generic Hankel matrix by setting to $0$ at most $e-1$ of the variables in that matrix.
  Hence the result mentioned above applies,  and  we conclude that  $I_{e+1}(P_T)$ is prime of codimension $m-d-e$. This particular  application of Eisenbud's
  result on $1$-generic matrices to special coordinate sections of generic Hankel matrices appears in more explicit form in \cite[Proposition 2.4]{CMRS}. 
  \end{proof}

\begin{corollary}
  \label{cor:univ}
  The Taylor variety $\mathcal{T}^1_{d,e,m} \subset \PP^m$ is irreducible of dimension $\min\{d+e,m\}$.
  If $d+e<m$, then $\mathcal{T}^1_{d,e,m} = \{\,T \in \PP^m \,\,:\,\, {\rm rank}(P_T) \leq e\,\}\,$
  and its prime ideal equals $I_{e+1}(P_T)$.
\end{corollary}

\begin{proof}
The variety  $\mathcal{T}^1_{d,e,m}$ is irreducible, as it is the closure of the image
of a polynomial~map. If $d+e \geq m$, then $\mathcal{T}^1_{d,e,m} = \PP^m$ by 
\Cref{lem:dense},  and we are done. We thus  assume $d+e < m$.

    \Cref{thm:ideal:minors} tells us that  $V({I_{e+1}(P_T)})$ is irreducible.
  We claim that it equals   $ \mathcal{T}^1_{d,e,m}$.
  If $T \in\psi(\CC^d \times \CC^e)$ then there exists
  $Q \in \CC[x]_{\leq e}$ with $Q(0)=1$ and the terms of $QT$ of degree
  from $d+1$ to $m$ vanish. In matrix notation, $P_T \, (1,q_1,\ldots,q_e)^t = {\bf 0}^t$.
  Hence the $(e+1) \times (e+1)$ minors of $P_T$ are zero.
  Therefore, $V(I_{e+1}(P_T))$  contains
 $\psi(\CC^d \times \CC^e)$ and its closure
  $\mathcal{T}^1_{d,e,m}$.
  
  For the converse, we will show that
    $\mathcal{T}^1_{d,e,m}$ contains an open subset of $V(I_{e+1}(P_T))$.
      Let $\mathcal{C} \subset V(I_{e+1}(P_T))$ be the closed set of 
  all $T$ such that
  the last $e$ columns of $P_T$ are linearly dependent. 
  Then $\mathcal{V} = V(I_{e+1}(P_T)) \backslash \mathcal{C}$
  is  non-empty and open in $V(I_{e+1}(P_T))$.
   If $T \in \mathcal{V}$ then the kernel of $P_T$ has a vector with first coordinate $1$.
   Hence $T \in\psi(\CC^d \times \CC^e)$, and
    $\mathcal{V} \subset
  \mathcal{T}^1_{d,e,m}$.
We conclude that $V({I_{e+1}(P_T)}) = \mathcal{T}^1_{d,e,m}$, and hence
  $\,I(\mathcal{T}^1_{d,e,m}) = \sqrt{{I_{e+11}(P_T)}} = {I_{e+1}(P_T)}$.
  \end{proof}

\begin{example}[$ e = d+1$] \label{ex:isoneless}
The degree of the numerator is one less than that of the denominator.
Here the Taylor variety $\mathcal{T}^1_{d,d+1,m}$ is a
secant variety of the rational normal curve,~and
  $$
  P_T \,\, = \,\,
    \begin{bmatrix}
    c_m     & c_{m-1} & \cdots & c_{m-d-1} \\
    c_{m-1} & c_{m-2} & \cdots & c_{m-d-2}   \\
    \vdots  & \vdots  &   \ddots     & \vdots    \\
    c_{d+1} & c_{d}   & \cdots &   c_0     \\
  \end{bmatrix}
  $$
  is the catalecticant in degrees $(m-d-1,d+1)$ of the binary form
  $f = \sum_{i=0}^m \binom{m}{i} c_i X^{i} Y^{m-i}$.
In symbols, $\mathcal{T}^1_{d,d+1,m} = \sigma_{d+1}(v_{1,m}(\PP^1)) \, = \,\{T : 
\text{rank}(P_T) \leq d+1\}$. This variety comprises
   binary forms $f$ of degree $m$ that are sums of  $d+1$
  $m$-th powers of linear forms in $X$ and $Y$.

  When $m=2d+2$, we get a hypersurface.
For instance, for $m=4,d=1$, the hypersurface  $\mathcal{T}^1_{1,2,4} \subset \PP^4$
is defined by the $3 \times 3$ Hankel matrix. It
comprises quartics
  $f=c_0Y^4+4c_1XY^3+6c_2X^2Y^2+4c_3X^3Y+c_4X^4$ that are
  sums of at most two fourth powers of linear forms.
    \exend
\end{example}

We conclude by summarizing the four cases
of what the Taylor variety for $n=1$ can be:
\begin{itemize}
\item
  If $e+d\geq m$ then  $\mathcal{T}^1_{d,e,m}$ is equal to $\PP^m$.
\item
  If $e+d< m$ and $e=d+1$ then $\,\mathcal{T}^1_{d,d+1,m}=\sigma_{d+1}(v_{1,m}(\PP^1))
  \subset \PP^m\,$
  is the $(d+1)$-st secant variety to the rational normal curve of degree $m$. 
  \item
  If $e+d< m$ and $e<d+1$,  then our Taylor variety is a cone with apex $\PP^{d-e}$, namely
$$ \mathcal{T}^1_{d,e,m} \,\,=\,\, \mathcal{T}^1_{e-1,e,D} \,\star\, \PP^{d-e}
  \,\, =\,\, \sigma_{e}(v_{1,D}(\PP^1)) \,\star\, \PP^{d-e} \,\, \subset \,\,\PP^m $$
  with $D=m-d+e-1$. 
  This is like the last case, with coordinates $c_0,\ldots,c_{d-e}$ missing.
\item
If $e+d< m$ and $e>d+1$, then the Taylor variety is a linear section of the secant variety
above. Namely, setting  $D' = m+d-e+1$, we have
   $\mathcal{T}^1_{d,e,m} = \sigma_{e}(v_{1,D'}(\PP^1)) \cap H$ where $H \subset \PP^{D'}$
  is the linear space defined by the vanishing of $e-d-1$ coordinates.
  \end{itemize}
We conclude that the Taylor variety for $n=1$ is
a classical object that is well-understood.
In the next sections we shall see that the
situation is different and more interesting for $n \geq 2$.

\section{Pad\'e Matrix}
\label{sec3}

The Taylor variety $\mathcal{T}^n_{d,e,m} \subset \PP^{\binom{n+m}{n}-1}$
was defined by
the parametrization in (\ref{eq:Pade}). Section \ref{sec2} offered
an implicit representation  for $n=1$. In what follows, we generalize
this to $n \geq 2$.
Let $\CC[x]$ denote the polynomial ring in $n$ variables $x = (x_1,\ldots,x_n)$.
Fix integers $d,e,m \geq 0$.
Write $\CC[x]_{\leq e}$  for
the $\binom{n+e}{n}$-dimensional
space of polynomials of degree at most $e$,
and $M_{d+1,m}$ for the set of monomials in $\CC[x]_{\leq m} $ but not in $ \CC[x]_{\leq d}$.
We have $| M_{d+1,m} | =  \binom{n+m}{n} - \binom{n+d}{n} $.

Consider a polynomial $T=\sum_{|\gamma| \leq m} c_\gamma x^\gamma$
in $\CC[x]_{\leq m}$ with $T(0) = c_0 = 1$.
Let $\varphi_T$ denote the $\CC$-linear map defined by
the same formula as in (\ref{eq:varphi1}).
But that formula is now meant for a polynomial ring $\CC[x]$
in $n$ variables $x= (x_1,\ldots,x_n)$.
  The {\em Pad\'e matrix} $P_T$
 represents $\varphi_T$ with respect to the monomial bases.
The rows of $P_T$ are indexed by $M_{d+1,m}$,
ordered decreasingly by total degree, and the columns of $P_T$ 
are indexed by $M_{0,e}$, ordered increasingly by total degree.
The entry of $P_T$ in row $x^\alpha \in M_{d+1,m}$ 
and column $x^\beta \in M_{0,e}$ equals $c_{\alpha-\beta}$ if
$\beta \leq \alpha$ and  $0$ otherwise.
The Pad\'e matrix has format $(\binom{m+n}{n} - \binom{d+n}{n}) \times
\binom{e+n}{n}$, and it has a block Hankel structure.
That structure looks like  (\ref{eq:convention})
but now the blocks have different sizes.

\begin{example}[$n=2,d=e=2,m=5$] 
Just like in (\ref{eq:intro1}),
let us consider the fifth Taylor polynomial $T$ of two quadrics,
but now in two variables. The Pad\'e matrix has
the structure 
\begin{equation}
\label{eq:lookslike} P_T \,\, = \,\, 
 \begin{bmatrix}
C_5 & C_4' & C_3'' \\
C_4 & C_3' & C_2' \\
C_3 & C_2 & C_1  \end{bmatrix} .
\end{equation}
This looks like the matrix in (\ref{eq:intro3}), but $P_T$ has $15 = 6+5+4$ rows and
$6 = 1+2+3$~columns. Each block with index $i$ represents multiplication with
the degree $i$ component of $T$, {\it e.g.}~the $4 \times 2$ matrix $C_2:\CC[x]_1 \rightarrow \CC[x]_3$
and the $5 \times 3$ matrix $C_2':\CC[x]_2  \rightarrow \CC[x]_4$
are multiplication with the quadratic part of $T$.
The column $(C_5,C_4,C_3)^t$ represents
$\CC \rightarrow \CC\{M_{2,4}\}, c \mapsto cT$.
\exend
\end{example}

We now describe the block structure in general, since this will be important
later on. Let $T_i$ denote the homogeneous component in degree $i$ of the 
inhomogeneous polynomial $T \in \CC[x]_{\leq m}$.
For every degree $j$, multiplication by $T_i$ defines a $\CC$-linear map $\CC[x]_j \rightarrow \CC[x]_{i+j}$.
Let $C_i$ denote the matrix which represents the linear map $T_i$
with respect to the monomial~bases. From now on, we do not record the index $j$, but
we assume that $j$ is clear from the context.
The Pad\'e matrix $P_T$ is the aggregate of these blocks. Its block structure is independent of~$n$.
With this convention, the matrix $P_T$ in 
 (\ref{eq:lookslike}) now has
$C_4' \mapsto C_4$,
$C_3', C_3'' \mapsto C_3$ and
$ C_2' \mapsto C_2$.

\begin{example}[$d=3,e=4,m=5$] \label{ex:2by5}
For the fifth Taylor polynomial of cubic over quartic,
\begin{equation} \label{eq:2by5}  P_T \quad = \quad  \begin{bmatrix} 
C_5 & C_4 & C_3 & C_2 & C_1 \\
C_4 & C_3 & C_2 & C_1 & C_0 
\end{bmatrix}.
\end{equation}
The number of rows is $\binom{n+4}{5} + \binom{n+3}{4}$,
and the number of columns is 
$\,1 + n + 
\binom{n+1}{2} + 
\binom{n+2}{3} + 
\binom{n+3}{4} $.
This Pad\'e matrix has a distinguished
column vector in its kernel, written schematically as 
\begin{equation} \label{eq:2by5kernel} 
\begin{bmatrix} 
\,\,0 \,&\, 0\, & T_2 T_0 - T_1^2 &\, T_2 T_1- T_3 T_0 &\, T_3 T_1 - T_2^2
\, \end{bmatrix}^t.
\end{equation}
The five entries in (\ref{eq:2by5kernel})
are the coefficient vectors 
of the homogeneous components of the polynomial
$\, (T_2 T_0 - T_1^2) + (T_2 T_1- T_3 T_0) + (T_3 T_1 - T_2^2)$.
Do check that this is in the kernel of $\varphi_T$.
Our notation alludes to
 Cramer's rule for the 
$2\times 3$ submatrix on the right of $P_T$.
\exend
\end{example}

Let $\mathcal{I}^n_{d,e,m}$ denote the homogeneous prime ideal 
that defines the irreducible variety $\mathcal{T}^n_{d,e,m}$.
This is an ideal  in the polynomial ring $\CC[c]=\CC\bigl[c_\gamma : |\gamma| \leq m \bigr] $.
Our goal in this section is to 
determine the ideal $\mathcal{I}^n_{d,e,m}$
from the Pad\'e matrix $P_T$. We begin with some examples.

\begin{example}[$n=2,m=3$]
\label{ex:ternarycubics}
Here the ambient space is $\PP^9$ with coordinates $c_{00},c_{01},\ldots,c_{30}$.
The first two cases to consider are $(d,e) = (2,1)$ and $(d,e) = (1,2)$. The  Pad\'e matrices~are
$$
P_T  = 
\begin{small} \begin{bmatrix} C_3 & \! C_2 \end{bmatrix} = 
\begin{bmatrix}
 c_{30} & 0 & c_{20} \\
 c_{21} & c_{20} & c_{11} \\
 c_{12} & c_{11} & c_{02} \\
 c_{03} & c_{02} & 0
 \end{bmatrix}
  \end{small}
\,\,\,\, {\rm and} \quad P_T = 
\begin{small}
\begin{bmatrix} C_3 & \! C_2 & \! C_1 \\ C_2 & \! C_1 & \! C_0 \end{bmatrix} = 
\begin{bmatrix}
 c_{30} & 0 & c_{20} & 0 & 0 & c_{10} \\
 c_{21} & c_{20} & c_{11} & 0 & c_{10} & c_{01} \\
 c_{12} & c_{11} & c_{02} & c_{10} & c_{01} & 0 \\
 c_{03} & c_{02} & 0 & c_{01} & 0 & 0 \\
 c_{20} & 0 & c_{10} & 0 & 0 & c_{00} \\
 c_{11} & c_{10} & c_{01} & 0 & c_{00} & 0 \\
 c_{02} & c_{01} & 0 & c_{00} & 0 & 0
 \end{bmatrix}\! .
 \end{small}
$$
In both cases,  the prime ideal
$\mathcal{I}^2_{d,e,3}$ is generated by
the maximal minors of $P_T$, and it is
  Cohen-Macaulay of codimension~$2$.
We conclude that $\mathcal{T}^2_{2,1,3}$ has degree $6$ and 
its ideal $\mathcal{I}^2_{2,1,3}$
is generated by four cubics,
while $\mathcal{T}^2_{1,2,3}$ has degree $21$ and 
$\mathcal{I}^2_{1,2,3}$ is generated  by seven sextics.

The case $(d,e) = (1,1)$ is more interesting. The $4$-dimensional Taylor variety $\mathcal{T}^2_{1,1,3}$
represents Taylor cubics for ratios of two bivariate linear polynomials.
Its Pad\'e matrix~equals
$$ P_T \,\, = \,\,
\begin{bmatrix} C_3 &  C_2 \\ C_2 & C_1 \end{bmatrix} \,\,=\,\,
\begin{small}
\begin{bmatrix}
c_{30} & 0 & c_{20} \\
 c_{21} & c_{20} & c_{11} \\
 c_{12} & c_{11} & c_{02} \\
 c_{03} & c_{02} & 0 \\
 c_{20} & 0 & c_{10} \\
 c_{11} & c_{10} & c_{01} \\
 c_{02} & c_{01} & 0 \\
 \end{bmatrix}.
 \end{small}
$$
The ideal of maximal minors of $P_T$ is Cohen-Macaulay of 
expected codimension $5$ and degree $21$. But this ideal is not prime.
It is the intersection of the prime ideal  $\mathcal{I}^2_{1,1,3}$
with a primary ideal of degree $9$ whose radical is
$\langle c_{20},c_{11},c_{02},c_{10},c_{01} \rangle$.
Hence $\mathcal{T}^2_{1,1,3}$ has degree $12$ in $\PP^9$.
The ideal $\mathcal{I}^2_{1,1,3}$ is not Cohen-Macaulay. It is generated by
 $5$ quadrics, $16$ cubics and $1$ quartic.
\exend
\end{example}

Let $q = (q_\beta : \beta \in M_{0,e})$ be the column vector of coefficients of $Q$.
Then $S=\CC[c,q]$ is a polynomial ring in $\binom{n+m}{n}+\binom{n+e}{n}$ unknowns.
Let $J \subset S$ be the ideal generated by the entries of the column vector $P_T \cdot q$.
The variety $V(J)$ consists of pairs of polynomials $(T,Q)$ such that the 
product $TQ$ contains no monomials in $M_{d+1,m}$. 
The ideal saturation $(J:q_0^{\infty})$ describes the closure
of the pairs $(T,Q)$ such that
$Q(0) \not= 0$ and $TQ$ has no monomials in $M_{d+1,m}$
The projection of this variety $V \bigl( (J:q_0^{\infty})\bigr)$
onto the $T$-coordinates is the Taylor variety $\mathcal{T}^n_{d,e,m}$.
However, more is true: the projection gives the prime ideal 
that defines $\mathcal{T}^n_{d,e,m}$.

\begin{theorem}
  \label{IdealPrime}
The elimination ideal $\,  (J:q_0^{\infty}) \cap \CC[c]$
coincides with the prime ideal
    $\,\mathcal{I}^n_{d,e,m} $.
        \end{theorem}
    
  \begin{proof} 
  We already know that $  (J:q_0^{\infty}) \cap \CC[c]$ defines $\mathcal{T}^n_{d,e,m}$  as a set.
    Since elimination ideals of prime ideals are prime, it suffices to prove that  $(J:q_0^{\infty})$ is prime.
    Since $(J:q_0^{\infty})$ is the contraction to $S$ of the localization $JS_{q_0}$, we may as well prove that $JS_{q_0}$ is prime in $S_{q_0}$.    The generators of $J$ correspond to
  the rows of the Pad\'e matrix $P_T$, so they are indexed by $\alpha\in M_{d+1,m}$. 
    More precisely, for each
  $\alpha\in M_{d+1,m}$ we have a generator of $J$ of the form
  \begin{equation} 
\label{eq:ofW2} 
q_0c_{\alpha}+  \sum q_\beta c_\gamma \qquad   \mbox{ where the sum is over }
\{ (\beta, \gamma)\in M_{1,e}\times M_{0,m}  \,:\, \beta+\gamma=\alpha\}.
\end{equation}
In the quotient ring $S_{q_0}/JS_{q_0}$ we may use (\ref{eq:ofW2}) to write  $c_{\alpha}$ 
as a $\CC[q]$-linear combination of $c_\gamma$ with $|\gamma|<|\alpha|$. Therefore $S_{q_0}/JS_{q_0}$ is isomorphic to the polynomial ring  $\CC[c_\alpha : \alpha\in M_{0,d} ][ q]$ localized at $q_0$. 
Since this localization is a domain, we conclude that $JS_{q_0}$ is a prime ideal.   \end{proof} 
  
In our situation, the saturation can be carried out by
computing with non-homogeneous ideals, as follows.
Let $q'$ be the vector obtained from $q$ by setting the first coordinate $q_0$ to~$1$.

\begin{corollary} \label{cor:prime}
The prime ideal of the Taylor variety $\,\mathcal{T}^n_{d,e,m}$ 
can be obtained as follows:
\begin{equation}
\label{eq:idealelimination}
 \mathcal{I}^n_{d,e,m} \,\, = \,\,
 \bigl\langle P_T \cdot q'  \bigr\rangle \,\, \cap \,\, \CC[c].
\end{equation}
\end{corollary} 

\begin{proof}
Both ideals involve only $c$-variables,
and they are homogeneous in these variables.
We must therefore show that a homogeneous polynomial
$f(c)$ lies in $\mathcal{I}^n_{d,e,m}$
if and only if $f(c)$ is in the right hand side of (\ref{eq:idealelimination}).
We write $q' = [1,q'']^t$.
Suppose $ f(c)$ lies in $\mathcal{I}^n_{d,e,m}$.
By Theorem \ref{IdealPrime}, there exists $m \in \NN$
and a polynomial vector $g(c,q)$ 
such that  $ f(c) q_0^m = g(c,q)  P_T(c) q $.
Setting $q = 1$ in this identity, we see that
    $ f(c) = g(c,1,q'')  P_T(c) q'$
 lies in 
$ \bigl\langle P_T \cdot q'  \bigr\rangle \, \cap \, \CC[c]$.

Conversely, let $ f(c)$ be in 
$ \bigl\langle P_T \cdot q'  \bigr\rangle \, \cap \, \CC[c]$.
There is a row vector $ h(c,q') $ such that $   f(c) = h(c,q') P_T(c) q'$.
We divide each variable in $q'$ by $q_0$, and obtain        $ f(c) = h(c,q'/q_0)  P_T(c) (q'/q_0)$.
By clearing denominators, we obtain an identity         $ f(c) q_0^\nu = g(c,q)  P_T(c) q $
for some $\nu \in \NN$, where $g$   is the homogenization of $h$. 
 This shows that $f(c)$ lies in $\mathcal{I}^n_{d,e,m}$, by Theorem \ref{IdealPrime}.
\end{proof}

 Theorem \ref{IdealPrime} and Corollary \ref{cor:prime}
 show us how to get $\mathcal{I}^n_{d,e,m}$ in
 a computer algebra system. In light of
  \Cref{ex:ternarycubics}, one suspects that
 $\mathcal{I}^n_{d,e,m}$ can be obtained from the
 ideal $I_{\rm max}(P_T)$ of maximal non-vanishing minors of $P_T$ by saturation.
 This is presently only a conjecture.
To state it precisely,
let $\hat P_T$ be the matrix obtained from $P_T$ by deleting the leftmost column.
The corresponding linear map $\hat \varphi_T$ is the restriction of $\varphi_T$ to the
 subspace $\{ Q \in \CC[x]_{\leq e}: Q(0) = 0\}$. 

  \begin{conjecture} \label{conj:IJ}
 The prime ideal of the  variety $\mathcal{T}^n_{d,e,m}$ in
 the polynomial ring $\CC[c]$ satisfies
  \begin{equation}
\label{eq:idealsaturation}
 \mathcal{I}^n_{d,e,m} \,\, = \,\,
 \bigl(\, I_{\rm max} ( P_T) \,: \,I_{\rm max} (\hat P_T)^\infty \,\bigr) .
 \end{equation}
This is
the ideal of maximal non-vanishing minors of the Pad\'e matrix $P_T$, saturated by the
ideal of maximal non-vanishing  minors of the reduced Pad\'e matrix $\hat P_T$. 
See also Theorem~\ref{thm:dimm}.
\end{conjecture}

We close this section with an illustration of (\ref{eq:idealsaturation})
that sets the stage for what is to come.

\begin{example}[$n=2,d=1,e=2,m=4$] \label{ex:2124}
We consider rational functions in two variables $x$ and $y$ that are given as
the ratio of a linear polynomial and a quadratic polynomial:
$$ 
\frac{P(x,y)}{Q(x,y)} \,\, = \,\,
\frac{p_{00} + p_{10} x + p_{01} y}{
q_{00} + q_{10} x + q_{01} y + q_{20} x^2 + q_{11} xy + q_{02} y^2} \, .
$$
Our aim is to characterize the  quartic Taylor polynomials arising from such rational functions:
$$ \begin{matrix} T(x,y) & = & 
c_{00} \,+\, c_{10} x + c_{01} y \,+\, c_{20} x^2 + c_{11} xy + c_{02} y^2 + 
c_{30} x^3 + c_{21} x^2y + c_{12} xy^2\\  &. & \quad  \,+\, c_{03} y^3   \,
+ \,c_{40} x^4 + c_{31} x^3 y + c_{22} x^2 y^2 + c_{13} x y^3 + c_{04} y^4.
\end{matrix}
$$
The passage from $(P,Q)$ to $T$ defines a map $\,\PP^2 \times \PP^5 \dashrightarrow \PP^{14}$
that is birational onto its image, which is  the $7$-dimensional Taylor variety $\mathcal{T}^2_{1,2,4}$ in $\PP^{14}$.
We are interested in its ideal $\mathcal{I}^2_{1,2,4}$. 

The Pad\'e matrix  for this problem has format $12 \times 6$. It equals
\setcounter{MaxMatrixCols}{20}
$$ P_T \,\,=\,\,
\begin{bmatrix}
C_4 & C_3 & C_2 \\
C_3 & C_2 & C_1 \\
C_2 & C_1 & C_0
\end{bmatrix}
\,\, = \,\, \begin{small}
\begin{bmatrix} 
\, c_{40} & c_{31} & c_{22} & c_{13} & c_{04} &\,\, c_{30} & c_{21} & c_{12} & c_{03} & \,\, c_{20} & c_{11} & c_{02} 
\, \smallskip \\ \,
 c_{30} & c_{21} & c_{12} & c_{03} & 0 & \,c_{20} & c_{11} & c_{02} & 0 & \, c_{10} & c_{01} & 0 \, \\  \,
  0 & c_{30} & c_{21} & c_{12} & c_{03} & \, 0 & c_{20} & c_{11} & c_{02} & \, 0 & c_{10}  &  c_{01}
\,  \smallskip \\  \,
c_{20} & c_{11} & c_{02} & 0 & 0 & \,c_{10} & c_{01} & 0 & 0 &\,  c_{00} & 0 & 0 \, \\ \,
 0 & c_{20} &     c_{11} & c_{02} & 0 &\, 0 & c_{10} & c_{01} & 0 &\, 0 & c_{00} & 0 \,\\ \,
 0 & 0 & c_{20} & c_{11} & c_{02} &\, 0 & 0 & c_{10} & c_{01} & \, 0 & 0 & c_{00} \,\,
  \end{bmatrix}^t \!\! . \end{small}
 $$
This matrix has $924$ maximal minors of which $896$ are linearly independent. Thus,
the determinantal  ideal $I_{\rm max}(P_T)$ is generated by $896$ sextics in the $15$ unknowns $c_{ij}$.
This ideal has three associated primes.
One of them is the desired prime ideal $\mathcal{I}^2_{1,2,4}$ of dimension~$7$. 

The other two are nonreduced extraneous components.
The first one has dimension $8$, which exceeds ${\rm dim}(\mathcal{T}^2_{1,2,4}) = 7$. It is the linear subspace
$\PP^8$ defined by $\langle c_{00},c_{10},c_{01},c_{20},c_{11},c_{02}\rangle$.
This arises from the $3 \times 3$ minors in the last three columns of $P_T$.
The other extraneous component comes from the Veronese surface
$\mathcal{T}^2_{0,1,3} \subset \PP^9$ given by
the third Taylor polynomial of $\,1/(q_{00} + q_{10} x + q_{01} y)$.
It is the join of the surface $\mathcal{T}^2_{0,1,3}$ with the $\PP^4$ of 
binary quartics $c_{40} x^4 + c_{31} x^3 y + c_{22} x^2 y^2 + c_{13} x y^3 + c_{04} y^4 $.
This join is a variety of dimension $7$ and degree $9$. 

We can use Corollary \ref{cor:prime} to compute the ideal $\mathcal{I}^2_{1,2,4}$,
but this is a challenging computation.
 We find that $\mathcal{I}^2_{1,2,4}$
has at least $1392$ minimal generators, namely
 seven cubics and respectively $365, 754, 266$ polynomials in degrees $6,7,8$.
 We do not know whether  $\mathcal{I}^2_{1,2,4}$ requires additional generators.
Numerical computations show that
the Taylor variety $\mathcal{T}^2_{1,2,4}$ has degree $ 326$.
\exend
\end{example}

\section{Hypersurfaces}
\label{sec4}

We here study Taylor varieties that
are hypersurfaces in their ambient projective space.
We start out
with a pair of Taylor varieties that are {\em defective},
{\it i.e.}~the inequality \eqref{eq:expecteddimension}
is strict, thus setting the stage for our study of the dimensions
of Taylor varieties in Section~\ref{sec5}.
Thereafter, we turn to the main topic in Section~\ref{sec4}, namely
Hessians of Taylor hypersurfaces.
We show that many of them vanish identically,
thereby contributing to a circle of ideas that has a long history
in algebraic geometry, going back to 
Gordan and Noether in the 19th century.

\smallskip

We begin with  defective instances.
By ``expected'' we will mean that equality holds in~(\ref{eq:expecteddimension}).

\begin{proposition} \label{prop:degenhyper}
For $n=3$, there exists a Taylor variety, namely $\mathcal{T}^3_{2,2,3}$, that is expected to be a hypersurface
but has codimension two, and there also exists a Taylor variety, namely
$\mathcal{T}^3_{8,5,9}$, that is
expected to fill its ambient projective space but turns out to be a hypersurface.
\end{proposition}

\begin{proof}
For the proof we present explicit instances with $n=3$, and we verify the asserted properties.
Our first example has parameters $d=e=2$ and $ m=3$. The Taylor variety
$\mathcal{T}^3_{2,2,3}$ has expected dimension $18$
and it lives in the $\PP^{19}$ of cubic surfaces. Its Pad\'e matrix equals
 \begin{equation}\label{eq:c3223}
  P_T \,\, =\,\,
  \begin{small} \begin{bmatrix}
C_3 \!&\! C_2 \!&\! C_1
\end{bmatrix}   \,\,=\,\,
  \begin{bmatrix}
\,c_{300} & 0 & 0 & c_{200} & 0 & 0 & 0 & 0 & 0 & c_{100} \\
\,c_{210} & 0 & c_{200} & c_{110} & 0 & 0 & 0 & 0 & c_{100} & c_{010} \\
\,c_{201} & c_{200} & 0 & c_{101} & 0 & 0 & 0 & c_{100} & 0 & c_{001} \\
\,c_{120} & 0 & c_{110} & c_{020} & 0 & 0 & c_{100} & 0 & c_{010} & 0 \\
\,c_{111} & c_{110} & c_{101} & c_{011} & 0 & c_{100} & 0 & c_{010} & c_{001} & 0 \\
\,c_{102} & c_{101} & 0 & c_{002} & c_{100} & 0 & 0 & c_{001} & 0 & 0 \\
\,c_{030} & 0 & c_{020} & 0 & 0 & 0 & c_{010} & 0 & 0 & 0 \\
\,c_{021} & c_{020} & c_{011} & 0 & 0 & c_{010} & c_{001} & 0 & 0 & 0 \\
\,c_{012} & c_{011} & c_{002} & 0 & c_{010} & c_{001} & 0 & 0 & 0 & 0 \\
\,c_{003} & c_{002} & 0 & 0 & c_{001} & 0 & 0 & 0 & 0 & 0 
\end{bmatrix}. \end{small}
  \end{equation}
  The matrix is square of format $10 \times 10$, so we expect its
  determinant to define a hypersurface. But the matrix
        has rank $9$, so $f = {\rm det}(P_T)$ is zero.
  Its kernel of $P_T$ is generated by the vector
  \begin{equation}
  \label{eq:rightkernel}
  \begin{bmatrix} 0 & -C_1 & C_2 \end{bmatrix}^t \,\,=\,\,
    \bigl[\,0 \,,\,-c_{001},-c_{010},-c_{100},\,\,c_{002},\,c_{011},\,c_{020},\,c_{101},\,c_{110},\,c_{200}\, \bigr]^t .
   \end{equation}
This kernel is explained by
Theorem \ref{thm:linearsyzygy}.
Note that its first coordinate is zero.
The existence of a kernel vector with nonzero first coordinate
imposes a codimension $2$ constraint on $P_T$.
The variety $\mathcal{T}^3_{2,2,3}$ has dimension $17$ and degree $35$ in $\PP^{19}$.
Its ideal $\mathcal{I}^3_{2,2,3}$ can be computed via Theorem~\ref{IdealPrime}, and it
also verifies Conjecture \ref{conj:IJ}.
Here $ I_{\rm max} ( P_T) $ is the ideal of $9 \times 9$ minors
of the $ 10 \times 10$ matrix $P_T$, which has $81$ minimal generators, and
$I_{\rm max} (\hat P_T)$ is the ideal of $8 \times 8$ minors of the $10 \times 9$ matrix $\hat P_T$,
which has $315$ minimal generators.
The prime ideal $ \mathcal{I}^3_{2,2,3} $ is generated by ten octics.
It is not Cohen-Macaulay, but has Betti sequence $10,10,1$.

\smallskip

We next prove the second assertion by exhibiting an unexpected Taylor hypersurface.
Let $d=8,e=5$ and $m=9$. The two numbers on the right hand side of 
(\ref{eq:expecteddimension})
are both $219$, so  we expect $\mathcal{T}^3_{8,5,9}$ 
to precisely fill $\PP^{219}$. The Pad\'e matrix $P_T $ has $55$ rows and $56$ columns,
and its generic rank is $55$. The kernel of $P_T$ is spanned by a vector whose
whose first $20$ coordinates are zero and whose last $36$ coordinates
are linear in the $c_{ij}$. In particular, the determinant of the $55 \times 55$ matrix $\hat P_T$ is zero.
Among the remaining $55$ maximal minors of $P_T$, the first $19$ are zero
and the other $36$ are non-zero. The latter share a common factor $f$,
which is an irreducible polynomial of degree $54$. Each such maximal minor equals 
$f$ times a linear form, namely the corresponding coordinate in the kernel vector.
In conclusion, the Taylor hypersurface 
$\mathcal{T}^3_{8,5,9}$ has degree $54$ in $\PP^{219}$,
and it is defined by the equation $f=0$.

We obtain a second unexpected hypersurface by swapping
$d$ and $e$. This operation is the duality of reciprocal pairs,
to be explained in  Proposition \ref{prop:reciprocal}.
We now take
$d=5,e=8$ and $m=9$, where the Pad\'e matrix
has $164$ rows and $165$ columns.
Again, the inequality in (\ref{eq:expecteddimension})
is strict. We have ${\rm dim}(\mathcal{T}^3_{5,8,9}) = 218$,
while the expected dimension is $ \,{\rm min}\{219,219\} = 219$.
\end{proof}

\smallskip

For the rest of this section we assume that 
$n,d,e,m$ are parameters for which the matrix $P_T$ is square and
non-singular.
We abbreviate
 $f=\det(P_T)$, we write
$H_f$ for the {\it Hessian matrix} of $f$, and
$h_f = \det(H_f)$ for the {\it Hessian}. 
Already for $n=1$, these
Hessians are special.

\begin{example}[Hankel cubic revisited]
  \label{1124revisited}
Let $n=1, d=1, e=2$ and $m=4$.
Following Example~\ref{ex:isoneless},
the Taylor variety $\mathcal{T}^1_{1,2,4}$ is the cubic threefold in $\PP^4$ 
with defining polynomial
$$  f \,\, = \,\,\det  \begin{small}
  \begin{bmatrix}
    c_4 & c_{3} & c_{2} \\
    c_3 & c_{2} & c_{1} \\
    c_2 & c_{1} & c_{0} \\
  \end{bmatrix}.
\end{small}
$$
The Hessian of $f$ is the quintic $h_f = -8(c_0c_4 - 4 c_1c_3 + 3 c_2^2) f$.
Hence the Hessian vanishes on the hypersurface  $V(f) = \mathcal{T}^1_{1,2,4}$.
By \cite[Proposition 7.2.3]{Ru}, this shows that this cubic
has zero Gaussian curvature,  a noteworthy property in differential geometry.
Note that $\mathcal{T}^1_{d,2,d+3}$ has the same Pad\'e matrix as $\mathcal{T}^1_{1,2,4}$,
for all $d \geq 1$, so it also has zero Gaussian curvature.
\exend
\end{example}

The property in \Cref{1124revisited} is shared by
many Taylor varieties for $n=1$.
We conjecture that $\mathcal{T}^1_{d,e,d+e+1}$ has zero Gaussian
curvature for every $(d,e) \geq (1,2)$, see also \cite[Conjecture~3.3]{CMRS}.
Here we may assume $e \geq d+1$, because
the varieties $\mathcal{T}^1_{d,d+1,2d+2}$ ($e=d+1$) and 
$\mathcal{T}^1_{d+t,d+1,2(d+t)+2}$ have the same Padé matrix up to renaming variables, for every
$t \geq 0$. The case $d=0$ for $n=1$ is called {\it sub-Hankel} in \cite{CRS}:
here, by \cite[Theorem 4.4(iii)]{CRS}, the Hessian $h_f$  is a multiple of a power of $c_0$, so
the Hessian of $\mathcal{T}^1_{0,e,e+1}$
vanishes only at infinity. It is shown in \cite[Theorem~3.1]{CMRS}
that the Hessian of $\mathcal{T}^1_{d,e,d+e+1}$ does not vanish identically.

We now turn to the case $n \geq 2$. Here the situation is quite different:
Taylor hypersurfaces can have vanishing Hessian, {\it i.e.}~$h_f \equiv 0$.
This holds when
 $m=d+1$, as shown~in \Cref{vanishing_hessian}, but it holds in  other cases as well.
The minimal example of a Taylor hypersurface with vanishing Hessian is a cone over
the cubic surface found in 1900 by Perazzo  \cite{P}.

\begin{example}[$n=2,d=1,e=1,m=2$] \label{ex:2425p}
  \label{perazzo}
  The Taylor variety $\mathcal{T}^2_{1,1,2}$ is a cubic hypersurface. The
  Padé matrix  $P_T$ and
  the Hessian matrix of $f=\det(P_T)$ are
  $$
  P_T
 \, = \,
  \begin{bmatrix}
    c_{20} & c_{10} & 0 \\
    c_{11} & c_{01} & c_{10} \\
    c_{02} &  0 & c_{01}
  \end{bmatrix}
  \quad {\rm and} \quad
  H_f \,=\,\begin{small}
  \begin{bmatrix}
\,    \,\,2c_{20} & -c_{11} & 0       & -c_{10} & \,2c_{01}\, \\
    \,-c_{11} & \, 2c_{02} & 2c_{10} & -c_{01} &      0  \\
     \,     0 & \,2c_{10} & 0       & 0       & 0       \\
  \,  -c_{10} & -c_{01} & 0       & 0       & 0       \\
   \,\,\, 2c_{01} &  0      & 0       & 0       & 0       \\
  \end{bmatrix}. \end{small}
  $$
  We see that the rank of $H_f$ is $4$. Hence $h_f$ is the zero polynomial.
  The Taylor variety $\mathcal{T}^2_{1,1,2}$ is a cone over the \emph{Perazzo variety} \cite{P}.
  We note that $\mathcal{T}^2_{1,1,2}$ is the orbit closure of
  the action of  a $5$-dimensional group, consisting of
  ${\rm SL}(2)$ and the affine group $(\CC^2,+)$, which acts via
  $$  \qquad c_{02} \mapsto c_{02} + \alpha c_{01} ,\,
   c_{11} \mapsto   c_{11} + \alpha c_{10} + \beta c_{01},\,
 c_{20} \mapsto c_{20} + \beta c_{10} 
\qquad {\rm for} \,\,\,(\alpha,\beta) \in \CC^2. $$
Such an action exists for a large class of
Taylor hypersurfaces, as shown in \Cref{vanishing_hessian}. \exend
\end{example}

\begin{theorem}
  \label{vanishing_hessian}
  Fix $n,d,e \in \NN$ with $n\geq 2$. Let
$w$ be the vector of coefficients $c_\gamma$ seen in~$P_T$.
Suppose
  $\mathcal{T}^n_{d,e,d+1} = V(f)$,
    where $f = \det P_T$, and set
  $f_\gamma = \frac{\partial f}{\partial c_\gamma}$.
   The image of the~polar map
  $w \mapsto (f_\gamma : c_\gamma \in w)$
  is not dense in $\PP^{|w|-1}$, so the hypersurface
  $\mathcal{T}^n_{d,e,d+1}$ has  vanishing Hessian.
\end{theorem}

\begin{proof}
  Let $T = \sum_{|\gamma| \leq d+1} c_\gamma x^\gamma = T_0+T_1+\cdots+T_{d+1}$.
  The Padé matrix has the form
  $$
  P_T \,\,=\,\, [C_{d+1} \,\,\,\,\, C_d \,\,\,\,\, C_{d-1} \,\, \cdots \,\, C_{d-e+1}],
  $$
  where $C_j : \CC[x]_{d+1-j} \to \CC[x]_{d+1}$ is multiplication by $T_j$.
  The columns of $P_T$ are labeled by the monomials in $Q$, and these are sorted in 
  ascending degree lexicographic term order.
  
  For $j = d-e+2,\ldots,d+1$, let $\lambda^j = (\lambda^j_\alpha)_{\alpha \in M_{d_j,d_j}}$
  be a vector of new variables whose entries are indexed by
  the monomials of degree $d_j=j-(d-e+1)$ in $n$ variables.
  We now replace the first column
  $C^1_j$ of $C_j$ with $C^1_j + C_{d+1-e} \, \left[\begin{smallmatrix} \lambda^j &
      {\bf 0} \end{smallmatrix}\right]^T$, where ${\bf 0}$ is a row vector of zeroes of length
  $\binom{n+e-1}{e}-\binom{n+d_j-1}{d_j}$. This arises from the following affine action on 
  the vector $w$:
  \begin{equation*}
    \begin{array}{rclr}
      c_\gamma & \mapsto & c_\gamma & \text{ for } |\gamma| = d+1-e, \\
      c_\gamma & \mapsto & c_\gamma
      \, +\displaystyle\sum_{\alpha+\beta=\gamma} \lambda_\alpha^{|\gamma|} c_\beta & \text{ for } |\gamma| > d+1-e .
  \end{array}
  \end{equation*}
In this sum we have
  $\alpha \geq 0$ and $|\beta| = d-e+1$.
  The other columns of $C_j$ are modified by this action, for every $j$.
  By the structure of the block $C_j$, this again corresponds to column operations where
  $C^i_j$ is replaced by $C^i_j + C_{d+1-e} \, \left[\begin{smallmatrix} {\bf 0}_1 & \lambda^j & {\bf 0}_2
    \end{smallmatrix}\right]^T$, for zero row vectors ${\bf 0}_1$ and ${\bf 0}_2$.

  Let $f' \in \CC[x,\lambda]$ be the determinant of the matrix obtained by these 
  elementary operations. We have $f'=f$. Taking derivatives with respect to
   $\lambda^j$, where $ j=d-e+2,\ldots,d+1$, 
  one gets that the partial derivatives $f_\gamma$ of $f = {\rm det}(P_T)$
    satisfy the following
  linear equations:
    \begin{equation}
    \label{relations}
  0\, \,= \,\,\frac{\partial f}{\partial \lambda^j_\alpha} \,\,=\,\, \frac{\partial f'}{\partial \lambda^j_\alpha}
\,\,  = \! \sum_{c_\beta \in C_{d-e+1}} \!\!\! c_\beta  \cdot f_{\alpha+\beta} \quad \text{ for }
  d-e+2 \leq j \leq d+1 \text{ and } \alpha \in M_{d_j,d_j}.
  \end{equation}
  For $d-e+2 \leq j \leq d+1$, we define a Hankel matrix $M_j$ as follows: the $(\alpha,\beta)$ entry
  of $M_j$ is
  $f_{\alpha+\beta}$ for $c_\beta \in C_{d-e+1}$ and $\alpha$ such that $|\alpha+\beta| = j$.
  Next let $M = \left[M_{d+1} \, M_{d} \, \cdots \, M_{d-e+2}\right]^T$ be the block-Hankel matrix
  obtained by stacking all matrices $M_j$. By \eqref{relations}, the right kernel of
  $M$ is nonzero, so the columns are linearly dependent.
     Note that $M$ has exactly
  $$ 
  \sum_{j=d-e+2}^{d+1} \binom{n+d_j-1}{d_j} \,\,= \,\,\sum_{j=d-e+2}^{d+1} \binom{n+j-d+e-2}{j-d+e-1}
  \,\,=\,\, \sum_{j=0}^{e-1} \binom{n+j}{n-1}\,\, = \,\,\binom{n+e}{n}-1
  $$
  rows and $\binom{n+d-e+1-1}{d-e+1}  = \binom{n+d-e}{n-1}$ columns.
  Since $\mathcal{T}^n_{d,e,d+1}$ is a hypersurface,  \eqref{eq:expecteddimension}~implies
    $$
  \dim \mathcal{T}^n_{d,e,d+1} \,\,= \,\,\binom{d+1+n}{n}-2 \,\,\leq \,\,\binom{d+n}{n}+\binom{e+n}{n}-2.
  $$
  From this we obtain
  $\binom{e+n}{n} \geq \binom{d+1+n}{n}-\binom{d+n}{n}
  = \binom{d+n}{n-1}$. In particular, we find that $e \geq 1$, and hence
  $\,
  \binom{e+n}{n}-1 \geq \binom{d+n}{n-1}-1 > \binom{d-e+n}{n-1}-1.
  $
  We deduce that the number of rows of the block-Hankel matrix $M$ 
  is greater than or equal to  the number of columns of $M$.
    
  Since $M$ has a nonzero right kernel, its maximal minors must vanish.
  Thus the image of the polar map $w \mapsto (f_\gamma : \gamma \in w)$ lies in the variety
  $V(I_{\rm max}(M)) \subset \PP^{|w|-1}$. Since the Hessian matrix $H_f$ is the Jacobian
  of the polar map, its determinant $h_f$ vanishes identically.
\end{proof}

%We show the action described in the proof of \Cref{vanishing_hessian} for $\mathcal{T}^2_{4,2,5}$
%in \Cref{ex:2425}.

\begin{example}[$n=2,d=4,e=2,m=5$] \label{ex:2425}
This concerns quintic Taylor polynomials of
 bivariate rational functions of the type quartic divided by quadric.
The Taylor variety $\mathcal{T}^2_{4,2,5}$ is the sextic hypersurface in $\PP^{20}$
defined by the determinant $f$ of the Padé matrix
\begin{equation}
  \label{eq2425}
  P_T
  \,\, =\,\,
  \begin{bmatrix}
    C_5 & C_4 & C_3
  \end{bmatrix}
  \,\, = \,\,
  \begin{bmatrix}
    c_{50} & c_{40} & 0      & c_{30} & 0      & 0      \\
    c_{41} & c_{31} & c_{40} & c_{21} & c_{30} & 0      \\
    c_{32} & c_{22} & c_{31} & c_{12} & c_{21} & c_{30} \\
    c_{23} & c_{13} & c_{22} & c_{03} & c_{12} & c_{21} \\
    c_{14} & c_{04} & c_{13} &      0 & c_{03} & c_{12} \\
    c_{05} & 0      & c_{04} &      0 &      0 & c_{03} \\
  \end{bmatrix}
\end{equation}
Only $15$ of the $21$ unknowns $c_{ij}$ appear in $P_T$. The
reduced $15 \times 15$ Hessian matrix $H_f$ has generic corank two, and hence
$h_f \equiv 0$.
To illustrate the proof of \Cref{vanishing_hessian}, we write
$$
f \,\,=\,\, {\rm det}(P_T) \,\,\,\, =\,\,\,\,f' \,\, = \,\,
\det
\begin{bmatrix} \,
  C_5 + C_3
  \left[\begin{smallmatrix}
      \lambda^5_{20} \\
      \lambda^5_{11} \\
      \lambda^5_{02} \\
  \end{smallmatrix}\right] &
  C_4 + C_3
  \left[\begin{smallmatrix}
      \lambda^4_{10} &              0 \\
      \lambda^4_{01} & \lambda^4_{10} \\
                   0 & \lambda^4_{01} \\
  \end{smallmatrix}\right] &
  C_3 \,
\end{bmatrix}.
$$
By taking partical derivatives with respect
to $\lambda^5 = (\lambda^5_{20},\lambda^5_{11},\lambda^5_{02})$ and $\lambda^4
= (\lambda^4_{10},\lambda^4_{01})$, we deduce 
in (\ref{relations})
that the maximal minors of the
following block-Hankel matrix must vanish:
\begin{equation}
  \label{eq:5minorsf}
  M \,:=\,
  \left[     \begin{array}{c}     M_5 \\     \hline     M_4     \end{array}     \right]
\,\,  = \,\,
  \left[
  \begin{array}{cccc}
    f_{50} & f_{41} & f_{32} & f_{23} \\
    f_{41} & f_{32} & f_{23} & f_{14} \\
    f_{32} & f_{23} & f_{14} & f_{05} \\
    \hline 
    f_{40} & f_{31} & f_{22} & f_{13} \\
    f_{31} & f_{22} & f_{13} & f_{04}
  \end{array}   \right].
\end{equation}
Note that the codimension of $V(I_{\rm max}(M))$ in $\PP^{14}$ is equal to the
generic corank of $H_f$.
\exend
\end{example}

We believe that the last statement is always true in the
cases covered by \Cref{vanishing_hessian}. To be precise, we conjecture that
the generic corank of $H_f$ equals $\binom{n+e}{n}-\binom{n+d-e}{n-1}$.
This is the
expected codimension of the ideal of maximal minors of the block-Hankel matrix $M$.

In Proposition \ref{prop:reciprocal} we shall present a general duality statement for
Taylor varieties under swapping $d$ and $e$.
That duality seems to be compatible with vanishing Hessians.

\begin{example}
[$n=2,d=2,e=4,m=5$]
  \label{ex:2245}
Applying duality to \Cref{ex:2425} leads us to study
    quintic Taylor polynomials for
 quadrics divided by quartics.
The Taylor variety $\mathcal{T}^2_{2,4,5}$ is a hypersurface of degree $15$ in $\PP^{20}$.
It is defined by the determinant of the $15 \times 15$ Pad\'e matrix $P_T$.
All $21$ coefficients $c_{ij}$ of the quintic appear in $P_T$, so the
Hessian $H_f$ is a $21 \times 21$ matrix.
This Hessian has generic rank $19$, so its
determinant $h_f$ vanishes to second order.
  \exend
\end{example}

The study of hypersurfaces with vanishing Hessian has a long tradition.
  Hesse conjectured that all such hypersurfaces are cones, but Gordan and Noether 
  found a counterexample.
    Perazzo \cite{P} achieved substantial progress in the case of cubics.
The topic continues to fascinate geometers until the present day.
We refer to Chapter~7 in the prize-winning monograph \cite{Ru}, 
to the recent article \cite{GRS}, and to the references in these sources.
Our study of Taylor polynomials led us naturally
to the novel instances that are reported here. However, it seems to us that
\Cref{vanishing_hessian} is just the tip of an iceberg that remains to be explored.

 \section{Dimension}
  \label{sec5}

The Taylor variety $\mathcal{T}^n_{d,e,m}$ is the closure of the image of the rational map
given by (\ref{eq:Pade}), which~is
\begin{equation}
\label{eq:morphism} \psi \,:\, \PP^{\binom{d+n}{n}-1} \times \PP^{\binom{e+n}{n}-1} \,\,
\dashrightarrow \,\, \PP^{\binom{n+m}{n}-1}, \,\,
(P,Q) \,\mapsto \, T .
\end{equation}
The three spaces in \eqref{eq:morphism} parametrize polynomials of degrees
 $d,e,m$ in $\CC[x] = \CC[x_1,\ldots,x_n]$.
 We shall present a formula for the dimension of $\mathcal{T}^n_{d,e,n}$.
 To this end, we revisit the multiplication map  $\hat \varphi_T$ from
  $\{ Q \in \CC[x]_{\leq e}: Q(0) = 0\}$ to $\CC \{ M_{d+1,m} \}$
  that is defined by~(\ref{eq:varphi1}). The matrix that represents the $\CC$-linear map 
  $\hat \varphi_T$ is the {\em reduced Pad\'e matrix} $\hat P_T$,
  which we obtain from $P_T$ by deleting the first column.
  Thus $\hat P_T$ has $\binom{m+n}{n} - \binom{d+n}{n}$ rows
  and $\binom{e+n}{n}-1$ columns.
  
\begin{theorem} \label{thm:dimm}
The dimension of the Taylor variety $\mathcal{T}^n_{d,e,m}$ equals
$\binom{d+n}{n}-1$ plus the rank of 
the reduced Pad\'e matrix $\hat P_T$ at a generic point of $\mathcal{T}^n_{d,e,m}$,
provided this sum is less than $\binom{n+m}{n}$.
\end{theorem}  

\begin{proof}
The dimension of $\mathcal{T}^n_{d,e,m}$  is  $\binom{d+n}{n} + \binom{e+n}{n}-2$
minus the dimension of the generic fiber of the map $\psi$.
Since $\hat P_T$ has $\binom{e+n}{n}-1$ columns, Theorem \ref{thm:dimm} is equivalent
to the claim that this fiber dimension equals the dimension  of $\,{\rm kernel}(\hat P_T)\,$ 
at a generic point of $\mathcal{T}^n_{d,e,m}$.

Fix a generic point $(P,Q)$ in the domain of (\ref{eq:morphism}) and set $T = \psi(P,Q)$.
We work in the affine chart where $P(0)=Q(0) = 1$.
Let $(P',Q')$ be any point in the~fiber $\psi^{-1}(T)$. Since $P'(0) = Q'(0) = 1$,
the constant term of $Q-Q'$ is zero, and the product $T(Q-Q') = P-P'$  contains no monomials in $M_{d+1,m}$.
Therefore, $Q-Q'$ lies in the kernel of $\hat \varphi_T$, so that 
$Q'$ is a point in the affine space $Q + {\rm kernel}(\hat \varphi_T)$.
Conversely, for any $Q'$ in $Q + {\rm kernel}(\hat \varphi_T)$,   there is a unique $P'$
such that $(P',Q') \in (\psi^n_{d,e,n})^{-1}(T)$. Namely, define
$P'$ to be the order $m$ truncation of $Q'T$.
We conclude that the fiber of $\psi$ over $T$ is birationally isomorphic to ${\rm kernel}(\hat \varphi_T)$. 
\end{proof}

Here is one more useful fact, namely the reciprocal duality of Taylor varieties.

\begin{proposition} \label{prop:reciprocal}
The Taylor varieties $\mathcal{T}^n_{d,e,m} $
and $\mathcal{T}^n_{e,d,m}$ are birationally isomorphic.
In particular, they have the same dimension inside their common ambient space $\,\PP^{\binom {n+m}{n}-1}$.
\end{proposition}

\begin{proof}
We obtain a rational map
$\,\PP^{\binom {n+m}{n}-1} \dashrightarrow \PP^{\binom {n+m}{n}-1},\,T \mapsto 1/T\,$
by taking reciprocals and truncating at order $m$. This map is a morphism
on the affine chart given by $T(0)=1$.
Clearly, this reciprocal map $T \mapsto 1/T$ is an involution, so it is birational.
Moreover, this involution takes $\mathcal{T}^n_{d,e,m}$ onto
$\mathcal{T}^n_{e,d,m}$, as it sends $P/Q$ to $Q/P$.
This completes the proof.
\end{proof}

Theorem \ref{thm:dimm} tells us that $\mathcal{T}^n_{d,e,m} $ is defective if and only if
the reduced Pad\'e matrix $\hat P_T$ has lower than expected rank, and 
 Proposition \ref{prop:reciprocal} implies that defective Taylor varieties come in
 pairs under swapping $d$ and $e$.
 The specific varieties seen in the proof of Proposition \ref{prop:degenhyper}
serve to illustrate both results.
In what follows, various families of
defective Taylor varieties~will be identified and explained.
This will  lead us to a famous problem
in commutative algebra, known as Fr\"oberg's Conjecture.
We begin by presenting a census of small defective cases.

\begin{table}[h]
\begin{center}
\begin{small}
\begin{tabular}{||c c c||} 
 \hline
 $n$ & $d,e,m$ & ${\rm dimensions}$ \\ [0.3ex] 
 \hline\hline
 3 & 2,2,3 & 17,18,19 \\ \hline
 3 & 3,4,5 & 52,53,55 \\ \hline
3 & 4,3,5 & 52,53,55 \\ \hline
3 & 4,6,7 & 116,117,119 \\ \hline
3 & 6,4,7 & 116,117,119 \\ \hline
3 & 5,8,9 & 218,219,219 \\ \hline
3 & 8,5,9 & 218,219,219 \\  \hline
4 & 2,2,3 & 27,28,34  \\ \hline
4 & 3,3,4 & 63,68,69  \\ \hline
4 & 3,4,5 & 102,103,125  \\ \hline
4 & 4,3,5 & 102,103,125  \\ [0.3ex]
\hline
\end{tabular} \qquad 
\begin{tabular}{||c c c||} 
 \hline
 $n$ & $d,e,m$ & ${\rm dimensions}$ \\ [0.3ex] 
 \hline\hline
4 & 4,4,5 & 122,138,125 \\ \hline
4 & 4,5,6 & 189,194,209 \\ \hline
4 & 5,4,6 & 189,194,209 \\ \hline
4 & 4,6,7 & 277,278,329 \\ \hline
4 & 5,6,7 & 318,334,329 \\ \hline
4 & 6,4,7 & 277,278,329 \\ \hline
4 & 6,5,7 & 318,334,329 \\ \hline
4 & 5,7,8 & 449,454,494 \\ \hline
4 & 6,6,8 & 417,418,494 \\ \hline
4 & 7,5,8 & 449,454,494 \\   \hline
5 & 2,2,3 & 39,40,55     \\ \hline
5 & 3,3,4 & 104,110,125     \\ [0.3ex]
\hline
\end{tabular}
\qquad 
\begin{tabular}{||c c c||} 
 \hline
 $n$ & $d,e,m$ & ${\rm dimensions}$ \\ [0.3ex] 
 \hline\hline
5 & 3,4,5 & 179,180,251     \\ \hline
5 & 4,3,5 & 179,180,251     \\ \hline
5 & 4,4,5 & 228,250,251     \\ \hline
5 & 4,5,6 & 370,376,461     \\ \hline
5 & 5,4,6 & 370,376,461     \\ \hline
5 & 5,5,6 & 441,502,461     \\ \hline
5 & 4,6,7 & 585,586,791     \\ \hline
5 & 5,6,7 & 690,712,791     \\ \hline
5 & 6,4,7 & 585,586,791     \\ \hline
5 & 6,5,7 & 690,712,791     \\ \hline
5 & 6,6,7 & 780,922,791  \\ [0.3ex]
    \hline
\end{tabular}
\end{small}
\caption{\label{tab:defective} Defective Taylor varieties: the third column lists
the dimension and number of parameters of $\mathcal{T}^n_{d,e,n}$,
followed by the dimension $\binom{n+m}{n}-1$ of the ambient projective space.
\vspace{-0.15in}}
\end{center}
\end{table}

\begin{proposition}
Table \ref{tab:defective} lists all defective Taylor varieties for $n \leq 5$ and $m \leq 12 -n$.
\end{proposition}

\begin{proof}
The $34 = 7 + 14+ 13 $ defective Taylor varieties for $n=3,4,5$ 
that are seen in Table~\ref{tab:defective} were found by exhaustive computation.
In each case, the dimension was computed as the rank of the Jacobian
of the parametrization (\ref{eq:morphism}), and it was verified using Theorem~\ref{thm:dimm}.
In particular, we found experimentally that no defective Taylor varieties exist for $n=2$.
\end{proof}

Table~\ref{tab:defective}
suggests that defectivity persists as the dimension grows:~if $\mathcal{T}^n_{d,e,m}$ is defective then so is $\mathcal{T}^{n+1}_{d,e,m}$. We also see
  five cases of Taylor varieties where the number of
 parameters exceeds the ambient dimension.
The smallest example is in the first row of the middle box.
% of Table~\ref{tab:defective}.

\begin{example}[$n=d=e=4,m=5$]
The Pad\'e matrix $P_T =
\begin{bmatrix} C_5 \!&\! C_4 \!&\! C_3 \!&\! C_2 \!&\! C_1 \end{bmatrix}
$ has $56$ rows and $70$ columns. Its rank is $54$, so the defect is $16$.
Subtracting $16$ from the expected dimension $138$, we obtain
$\,{\rm dim}(\mathcal{T}^4_{4,4,5}) = 122$. This equals
$\binom{8}{4} - 1 + 53$, as in Theorem \ref{thm:dimm}.
Thus, in the space  $\PP^{125}$ of  quartic fourfolds,
our Taylor variety $\mathcal{T}^4_{4,4,5}$
has codimension $3$.
\exend
\end{example}

Based on our computational experiments, we now formulate
a general conjecture.

\begin{conjecture} \label{conj:finite}
The following holds concerning the defectivity of Taylor varieties:
 \item[(1)]  For $n=2$, all Taylor varieties $\mathcal{T}^2_{d,e,m}$ are non-defective.
\item[(2)]  For $n=3$, there are only seven defective Taylor varieties, namely
 those listed in Table~\ref{tab:defective}.
\item[(3)] For $n\ge 3$ fixed,  there are only finitely many triples $(d,e,m)$
such that $\mathcal{T}^n_{d,e,m}$ is defective.  
\end{conjecture}

\section{Fr\"oberg's Conjecture}
\label{sec6}

We now turn to the promised relationship with commutative algebra. This will lend
support to Conjecture~\ref{conj:finite}. We examine the case $m=d+1$, which also appeared
in equation~(\ref{eq:rightkernel}).

\begin{theorem}\label{thm:linearsyzygy}
Consider the ideal generated by $e$ general forms
of degrees $d, d-1, \ldots ,d-e+1$ in $n$ variables.
The value of its Hilbert function in degree $d+1$ is the codimension of  $\,\mathcal{T}^{n}_{d,e,d+1}$.
\end{theorem}

\begin{proof} The ideal we are referring to is generated by the homogeneous components of $T$:
$$ I \,:= \, \langle \,T_d\,, \,T_{d-1}\,,\, T_{d-2}, \,\ldots,\, T_{d-e+1}\, \rangle \,\,\,\,
\subset \,\,\,\, \CC[x] \,=\,\CC[x_1,\ldots,x_n]. $$
We claim that $\binom{n+m}{n}-1 - {\rm dim}(\mathcal{T}^n_{d,e,d+1} )$ equals
 the Hilbert function value  $\,{\rm dim}_\CC \bigl( \CC[x]_{d+1}/I_{d+1} \bigr) $.

 In the case $m=d+1$, the Pad\'e matrix consists of 
only one row of blocks, namely it~equals
$$ \hat P_T \,\,=\,\, \begin{bmatrix}
C_{d} & C_{d-1} & \cdots &C_{d-e+3} & C_{d-e+2} & C_{d-e+1}
\end{bmatrix}. $$
The entries of $\hat P_T$ are the coefficients of general forms $T_d,T_{d-1},\ldots,T_{d-e+1}$.
We claim that the rank of this generic $\hat P_T$  equals the
rank of $\hat P_T$  at a generic point of $\mathcal{T}^n_{d,e,m}$.
This holds because the polynomial
$(T_{e}T _{d-e+1}+T_{d} +T_{d-1} + \cdots +T_2+T_1+1)(1-T_e)$
has no terms in degree $d+1$, which means that the left factor 
defines a point of $\mathcal{T}^n_{d,e,d+1}$, even for general $T$.

Combining the claim with Theorem \ref{thm:dimm},
we conclude that the dimension of $\mathcal{T}^n_{d,e,d+1}$
equals $\binom{d+n}{n}-1 + {\rm rank}(\hat P_T)$ for general $T$.
Now, any vector in the image of $\hat P_T$ corresponds to a homogeneous
polynomial $T_{d+1}$ that is in the ideal $I$.
In symbols, ${\rm rank}(\hat P_T) = {\rm dim}_\CC(I_{d+1}) $.
Hence the codimension of the Taylor variety $\mathcal{T}^n_{d,e,d+1}$
in its ambient space $\PP^{\binom{n+d+1}{n}-1}$ is equal~to
$$ 
{{n{+}d{+}1}\choose n}-1-{{d{+}n}\choose n}+1-\mathrm{rank}(\hat P_T)\,
= \,{{d{+}n}\choose{n{-}1}}-\mathrm{rank}(\hat P_T)
\,= \,{\rm dim}_\CC(\CC[x]_{d+1} ) - {\rm dim}_\CC (I_{d+1}). $$
The right hand side is the value of the Hilbert function in the assertion.
\end{proof}

The Hilbert function of an ideal of generic forms is a prominent thread
in commutative algebra, initiated by
Fr\"oberg's article \cite{Fr}. In our case, the Hilbert series is believed to be
\begin{equation}
\label{eq:froeberg_conjecture}
\left[ \frac{\prod_{i=1}^e(1-t^{d+1-i})}{(1-t)^n}  \right]_{>0}.
\end{equation}
The operator $[\dots ]_{>0}$ applied to a power series $\sum_{i\geq 0} a_i t^i$ returns $\sum_{i\geq 0} b_i t^i$  with $b_i=a_i$ if $a_j>0$ for all $j\leq i$ and $b_i=0$ otherwise. 
We are interested in the coefficient $t^{d+1}$ in~(\ref{eq:froeberg_conjecture}).
{\em Fr\"oberg's Conjecture} states that this coefficient equals
the Hilbert function value in Theorem~\ref{thm:linearsyzygy}.

The conjecture has been proved for several special cases.  
For $n=2$ it is due to Fr\"oberg \cite[Section 3, page 129]{Fr}.
 A simpler proof by Valla    \cite{Valla} rests on generic initial ideals.  For $n=3$ it was proved by  Anick \cite{A}. 
 Pardue   \cite{Pardue} showed that the conjecture is implied by the
 Moreno-Socias  Conjecture about the revlex generic  initial ideal of generic polynomials. 
 We refer to  recent papers by   Nenashev \cite{Ne} 
 and Nicklasson \cite{Ni}  for the state of the art  and many references.
 These known results allow for the derivation of special cases of Conjecture~\ref{conj:finite}.

\begin{theorem}\label{thm:FRob} 
Fr\"oberg's conjecture  implies  part (3) of Conjecture \ref{conj:finite}   when $m=d+1$.  
\end{theorem}
 
 The proof to be presented is based on suggestions of  Christian Krattenthaler. 
 We warmly thank  him for allowing us to include his nice combinatorial arguments
 in this paper.

\begin{proof}  Fix $n \geq 2$. For any $1\leq e\leq d$  we abbreviate
  $W(d,e)=\binom{d+n}{n-1}-\binom{e+n}{n}+1$ and  
$F(d,e)=(1-t)^{-n} \prod_{i=1}^e (1-t^{d+1-i} )$.  Assuming 
the correctness of Fr\"oberg's Conjecture, 
by  Theorem \ref{thm:linearsyzygy}, the  codimension of the Taylor variety  $\mathcal{T}^n_{d,e,d+1}$ is 
$$ \alpha(d,e)\,\,:=\, \mbox{ coeff. of } t^{d+1} \mbox{ in   } [ F(d,e) ]_{>0}.$$
On the other hand, the expected codimension of $\mathcal{T}^n_{d,e,d+1}$ equals
$$\beta(d,e)\,\,:=\,\,\max(0,  W(d,e)).$$ 
In particular, the following inequality holds for all 
values of $d$ and $e$:
\begin{equation}
\label{eq:ineq}
 \beta(d,e) \,\,\leq \,\, \alpha(d,e).
 \end{equation} 

We claim that $\alpha(d,e)=\beta(d,e)$ holds, with only finitely many exceptional
pairs $(d,e)$. Recall that $n$ is fixed.
We shall now prove this claim. Our argument proceeds in five steps:
\begin{enumerate}
\item[(1)] If   $\beta(d,e)=0$ then  $\beta(d,e+1)=0$. \vspace{-0.1cm}
\item[(2)] If    $\alpha(d,e)=0$ then  $\alpha(d,e+1)=0$.  \vspace{-0.1cm}
\item[(3)] If  $\beta(d,e)= \alpha(d,e)=0$  then $\beta(d,e_1)= \alpha(d,e_1)=0$  for all $e_1\geq e$.   \vspace{-0.1cm}
\item[(4)] If  $e<(d/2)+1$ then $\beta(d,e)= \alpha(d,e)$.  \vspace{-0.1cm}
\item[(5)] There exists $d_0$ such that $\alpha(d,e)=\beta(d,e)$ for every $d\geq d_0$ and for every $e\leq d$. 
\end{enumerate} 
 Assertion  (1) is obvious since $W(d,e)$ is a decreasing function in $e$. For (2), by assumption there is a $c\leq d+1$  such that  the coefficient  of $t^c$ in $F(d,e)$ is $\leq 0$. Let $c$  be the smallest integer with that property. Then 
 $\,F(d,e)=u_0+u_1t+\cdots+u_ct^c+\ldots\,$   with $u_c\leq 0$ and $u_i>0$ for $i=0,1,\dots,c-1$.  Since $F(d,e+1)=F(d,e) (1-t^{d+1-(e+1)})$,  the coefficient  of $t^c$ in  $F(d,e+1)$ is  $ u_c$  if $d-e>c$  or 
$ u_c-u_k$   if $d+1-(e+1)\leq c$  and $k=c-d+e$.  In both cases this value is  $\leq 0$ and so  $\alpha(d,e+1)=0$. 
Now  (3) follows immediately from (1) and (2). 
 
To prove (4), we first observe that the assumption $e<(d/2)+1$ implies that 
$$\prod_{i=1}^e (1-t^{d+1-i} )\,\,=
\,\,1-\sum_{i=1}^e t^{d+1-i} +  \dots  \mbox{ terms of degree } >d+1$$ 
Therefore, up to degree $d+1$, our series $F(d,e)$ coincides with the series
$$\left( \sum_{k\geq 0} \binom{n-1+k}{n-1} t^k\right)\left(1-\sum_{i=1}^e t^{d+1-i} \right).$$ 
It follows that the coefficient  of $\,t^{d+1} \,$ in  $\,F(d,e)\,$ equals
\begin{equation}
\label{eq:bigbinomials}
 \binom{n+d}{n-1}-\sum_{i=1}^e \binom{n-1+i}{n-1}\,\,=\,\,
 %   \binom{n+d}{n-1}-\sum_{i=0}^e \binom{n-1+i}{n-1}+1=\\
 \binom{n+d}{n-1}-\binom{n+e}{n}+1\,\,=\,\, W(d,e).
 \end{equation}
 If $W(d,e)\leq 0$ then both $\beta(d,e)$ and $\alpha(d,e)$ are zero.
If  $W(d,e)>0$ then $\beta(d,e)=W(d,e)$. Hence, by  (\ref{eq:ineq}), 
we have $\alpha(d,e)\geq \beta(d,e)>0$. In light of (\ref{eq:bigbinomials}),
we have $\alpha(d,e)=W(d,e)$ as well. This concludes our proof of assertion (4).  

 Now we prove assertion (5).  
  It follows from (3) and (4) that 
  the existence of a positive integer $e_1<(d/2)+1$ with $W(d,e_1)\leq 0$ 
  implies that $\beta(d,e)= \alpha(d,e)$ for all $e$.  Hence it suffices to show that there exists a  $d_0$ such that $W(d,e_1)\leq 0$  for all $d\geq d_0$ for some $e_1<(d/2)+1$. For $d=2p$ even, set $e_1=p$.
 This choice guarantees that, for degree reasons,
 \begin{eqnarray} 
 \label{Wroots1}
 W(2p,p)\,\,=\,\,\binom{2p+n}{n-1}-\binom{p+n}{n}+1
 \end{eqnarray}  
 is eventually negative as a function of $p$. For $d=2p-1$ odd, we take $e_1=p$ and observe~that 
 \begin{eqnarray} 
  \label{Wroots2}
 W(2p-1,p)\,\,=\,\,\binom{2p-1+n}{n-1}-\binom{p+n}{n}+1
  \end{eqnarray} 
is eventually negative, and assertion (5) follows.
This concludes the proof of
Theorem~\ref{thm:FRob}.
\end{proof}

 \begin{remark} 
 \label{expFR}
 The constant $d_0$ in step (5) above depends on $n$.
 For any given $n$, it can be computed explicitly  by looking at the roots of the polynomials  (\ref{Wroots1}) 
 and  (\ref{Wroots2}). For instance, for $n=4$ we find $d_0 =  144$.  Knowing $d_0$ and 
 assuming Fr\"oberg's Conjecture, we can quickly identify all pairs  $(d,e)$ such that  $\,\mathcal{T}^{n}_{d,e,d+1}$  is defective. Namely, for all $(d,e)$ with $e\leq d<d_0$, we check  whether   $\alpha(d,e)\neq \beta(d,e)$.
 For $n=4$, there are precisely $57$ defective pairs $(d,e)$:
$$ \begin{small} \begin{matrix}
    (2, 2), (3, 3), (4, 3), (4, 4), (5, 4), (6, 4), (6, 5), (7, 5), 
    (8, 5), (8, 6), (9, 6), (10, 6), (10, 7),
  (11, 7), \\ (11, 8), 
  (12, 7), (12, 8), (13, 8), (13, 9), (14, 8), (14, 9),
    (15, 9), (15, 10), (16, 9), (16, 10), (17, 10),\\  (17, 11), (18, 10),  
  (18, 11), (19, 11), (20, 11), (20, 12) ,
  (21, 12), (22, 12), (22, 13),  (23, 13), (24, 13), \\ (24, 14),(25, 14), 
 (26, 14), (26, 15), (27, 15), (28, 15), (28, 16) , (29, 16),  (30, 16),  (30, 17), \\ (31, 17), (32, 17), (33, 18),
  (34, 18), (35, 19), (36, 19), (37, 20), (38, 20), (40, 21), (42, 22).
\end{matrix}
\end{small}
$$ 
Note that the first eight of these $57$ pairs are listed in Table \ref{tab:defective}.
For $n=5$,  a computation shows that there are $431$ such exceptional pairs, with the largest one being
$(d,e) = (132,67)$.
 \end{remark}

\begin{corollary}
\label{cor(1)(2)}
Assume   $m=d+1$.  Then parts (1)  and (2) of Conjecture \ref{conj:finite}  are true.
\end{corollary}

\begin{proof} 
We know that Fr\"oberg's Conjecture holds for $n=2$ and $n=3$.
Then we can proceed as in Theorem \ref{expFR} and Remark \ref{expFR}.
We find $d_0 = 6$ for $n=2$, and $d_0 = 17$ for $n=3$.
The claim is established by listing all pairs $(d,e)$ with
$e \leq d < d_0$ that satisfy
$\alpha(d,e)\neq \beta(d,e)$.
\end{proof}

In conclusion, the discussion above connects
the study of Taylor varieties for $m=d+1$ 
to Fr\"oberg's longstanding conjecture on
ideals of general forms in $\CC[x]$.
A future project will 
extend this to arbitrary parameters $n,d,e,m$,
with ideals to be replaced by modules.
The modules arise from
the Hankel matrices in Section \ref{sec2}
but now the entries are homogeneous polynomials
$T_i$ of degree $i$ in $\CC[x] = \CC[x_1,\ldots,x_n]$.
Explicitly, the {\em Hankel matrix} is defined~as
\begin{equation}
\label{eq:H_T}
H_T \,\, := \,\, \begin{small}
\begin{bmatrix}
T_m & T_{m-1} & \cdots & T_{m-e} \\
T_{m-1} & T_{m-2} & \cdots & T_{m-e-1} \\
 \vdots & \vdots & \ddots & \vdots \\
T_{d+2} & T_{d+1} & \cdots & T_{d-e+2} \\
 T_{d+1} & T_{d} & \cdots & T_{d-e+1}
\end{bmatrix}. \end{small}
\end{equation}
The {\em Hankel module} is the graded module
generated by all columns of $H_T$ except the first~one.
The Taylor variety can be characterized by the constraint
that the first column lies in the Hankel module.
This suggests 
 an extension of Theorem \ref{thm:linearsyzygy}
from ideals to modules.
The study of Hankel modules and associated
vector bundles will be the topic of a follow-up article.

\bigskip
\bigskip

\footnotesize
\noindent {\bf Authors' addresses:}

\smallskip

\noindent Aldo Conca, 
Universit\`a di Genova,  
\hfill \url{conca@dima.unige.it}

\noindent Simone Naldi, Universit\'e de Limoges,  \hfill \url{simone.naldi@unilim.fr}

\noindent Giorgio Ottaviani,  Universit\`a degli Studi di Firenze,
 \hfill \url{giorgio.ottaviani@unifi.it} 
 
\noindent  Bernd Sturmfels, MPI-MiS Leipzig  and UC Berkeley \hfill \url{bernd@mis.mpg.de}

\end{document}